\documentclass[10pt,english,a4paper,reqno]{amsart}
\usepackage[T1]{fontenc}
\usepackage[latin9]{inputenc}
\usepackage{geometry}
\usepackage{amsthm}
\usepackage{amstext}
\usepackage{amssymb}

\makeatletter
\numberwithin{equation}{section}
\numberwithin{figure}{section}
\theoremstyle{plain}
\newtheorem{thm}{\protect\theoremname}[section]
  \theoremstyle{definition}
  \newtheorem{defn}[thm]{\protect\definitionname}
  \theoremstyle{plain}
  \newtheorem{prop}[thm]{\protect\propositionname}
  \theoremstyle{definition}
  \newtheorem{example}[thm]{\protect\examplename}
  \theoremstyle{remark}
  \newtheorem{rem}[thm]{\protect\remarkname}
  \theoremstyle{plain}
  \newtheorem{cor}[thm]{\protect\corollaryname}
  \theoremstyle{plain}
  \newtheorem{lem}[thm]{\protect\lemmaname}

\usepackage[all]{xy}

\makeatother

\usepackage{babel}
  \providecommand{\corollaryname}{Corollary}
  \providecommand{\definitionname}{Definition}
  \providecommand{\examplename}{Example}
  \providecommand{\lemmaname}{Lemma}
  \providecommand{\propositionname}{Proposition}
  \providecommand{\remarkname}{Remark}
\providecommand{\theoremname}{Theorem}

\begin{document}
\author{Adrien Dubouloz}
\address{Institut de Math\'ematiques de Bourgogne, Universit\'e de Bourgogne, 9 avenue Alain Savary - BP 47870, 21078 Dijon cedex, France} 
\email{adrien.dubouloz@u-bourgogne.fr}

\author{Alvaro Liendo}  
\address{Instituto de Matem\'atica y F\'\i sica, Universidad de Talca, Casilla 721, Talca, Chile}
\email{aliendo@inst-mat.utalca.cl}


\thanks{This research was supported in part by ANR Grant \textquotedbl{}BirPol\textquotedbl{}
ANR-11-JS01-004-0 and  by Fondecyt project 11121151.}

\title{On rational additive group actions}
\begin{abstract}
We characterize rational actions of the additive group on algebraic
varieties defined over a field of characteristic zero in terms of
a suitable integrability property of their associated velocity vector
fields. This extends the classical correspondence between regular
actions of the additive group on affine algebraic varieties and the
so-called locally nilpotent derivations of their coordinate rings.
This leads in particular to a complete characterization of regular
additive group actions on semi-affine varieties in terms of their
associated vector fields. Among other applications, we review properties
of the rational counter-part of the Makar-Limanov invariant for affine
varieties and describe the structure of rational homogeneous additive
group actions on toric varieties.
\end{abstract}
\maketitle

\section*{Introduction}

During the last decades, the systematic study of regular actions of
the additive group $\mathbb{G}_{a}$ on affine varieties has provided
very useful and effective tools to understand the structure of certain
of these varieties, most particularly those which are very close to
complex affine spaces from a topological or differential point of
view. One key feature of these actions in characteristic zero is that
they are uniquely determined by their associated velocity vector fields%
\footnote{This is no longer the case in positive characteristic where one has
to keep track of appropriate infinite collections of higher order
differential operators, see e.g. \cite{Mi68}%
} which, in turn, admit a very simple, purely algebraic characterization.
Namely, a global vector field on an affine $k$-variety $X=\mathrm{Spec}(A)$
is the same as a $k$-derivation $\partial$ of $A$ into itself,
and derivations corresponding to additive group actions are precisely
those with the property that $A$ is the increasing union of the kernels
of the iterated $k$-linear operators $\partial^{n}$, $n\geq1$.
Derivations $\partial$ with this property are called \emph{locally
nilpotent} and the co-morphism $\mu^{*}:A\rightarrow A[t]$ of the
corresponding $\mathbb{G}_{a}$-action $\mu:\mathbb{G}_{a}\times X\rightarrow X$
on $X$ is recovered by formally taking the exponential map 
\[
\exp(t\partial):A\rightarrow A[[t]],\qquad f\mapsto\sum_{n}\frac{\partial^{n}(f)}{n!}t^{n},
\]
and observing that the local nilpotency of $\partial$ guarantees
precisely that the latter factors through the sub-ring $A[t]$ of
$A[[t]]$. 

The study of affine algebraic varieties from a geometry point of view
benefited a lot a from the rich algebraic theory of locally nilpotent
derivations and therefore, it is very desirable to push further this
fruitful approach to more general settings. One possible direction
consists in re-interpreting the property for a global derivation $\partial$
of a ring $A$ of being locally nilpotent as a kind of ``algebraic
integrability condition'' through the above exponential map construction.
So given an arbitrary algebraic $k$-variety $X$ with field of rational
functions $K_{X}$ and a rational vector field $\partial$ on $X$,
viewed as a $k$-derivation $\partial:K_{X}\rightarrow K_{X}$, we
can again define formally the exponential map

\[
\exp(t\partial):K_{X}\rightarrow K_{X}[[t]],\qquad f\mapsto\sum_{n}\frac{\partial^{n}(f)}{n!}t^{n},
\]
and asks for counter-parts in this context of the previous integrability
condition. The most natural one, which we call rational integrability
(Definition~\ref{def:rat-int}), is to require that the previous
map factors through the sub-algebra $K_{X}(t)\cap K_{X}[[t]]$ of
$K_{X}[[t]]$. Our first main result (Theorem \ref{prop:RatIntegrable})
shows that rationally integrable rational vector fields on a variety
$X$ are in one-to-one correspondence with rational $\mathbb{G}_{a}$-actions
$\mathbb{G}_{a}\times X\dashrightarrow X$ on $X$. This notion also
turns out to coincide with the abstract algebraic notion of locally
nilpotent derivation of a field extension $K/k$ given by Makar-Limanov
\cite{MLNotes}, with the additional advantage that rational integrability
can be checked directly on generators of the field $K$ over $k$. 

Being local in nature, the rational integrability condition is much
more flexible than the property of being locally nilpotent, and this
enables the possibility to study local and global additional conditions
ensuring that a rational $\mathbb{G}_{a}$-action is actually regular.
For instance, we obtain a complete characterization of regular $\mathbb{G}_{a}$-actions
on semi-affine varieties $X$ in terms of their associated velocity
vector fields, viewed as $k$-derivations $\tilde{\partial}:\mathcal{O}_{X}\rightarrow\mathcal{O}_{X}$
from the structure sheave of $X$ to itself. Namely, we establish
(Theorem~\ref{thm:Regular-actions}) that regular $\mathbb{G}_{a}$-actions
on $X$ are in one-to-one correspondence with $k$-derivations $\tilde{\partial}:\mathcal{O}_{X}\rightarrow\mathcal{O}_{X}$
for which the induced $k$-derivations $\partial:K_{X}\rightarrow K_{X}$
and $\Gamma(X,\tilde{\partial}):\Gamma(X,\mathcal{O}_{X})\rightarrow\Gamma(X,\mathcal{O}_{X})$
of the field of rational functions and the ring of global regular
functions on $X$, are respectively rationally integrable and locally
nilpotent. In the case were $X$ is not semi-affine, these two conditions
are in general no longer sufficient to characterize regular $\mathbb{G}_{a}$-actions.
Nevertheless, they guarantee, thanks to a general construction due
to Zaitsev \cite{Za}, the existence of a partial completion of $X$
on which the rational $\mathbb{G}_{a}$-action on $X$ given by $\partial$
extends to a regular action. 

The last section of the article contains three applications of these
notions. The first concerns a generalization to the rational context
of the Makar-Limanov invariant \cite{MLNotes} and of its behavior
under stabilization. In our second application we give a combinatorial
description of homogeneous rational $\mathbb{G}_{a}$-actions on toric
varieties from which we derive a more conceptual proof of a characterization
of regular homogeneous $\mathbb{G}_{a}$-actions on semi-affine toric
varieties due to Demazure \cite{De}. The last application consists
of a characterization of line bundle torsors in terms of rational
$\mathbb{G}_{a}$-actions.

\section{Basic results on rational actions of the additive group\label{sec:1}}

In what follows, the term variety refers to a separated geometrically
integral scheme of finite type over a fixed base field $k$ of characteristic
zero. We denote by $\overline{k}$ an algebraic closure of $k$. An
algebraic group over $k$ is a group object in the category of $k$-varieties.
In particular, every algebraic group $G$ in our sense is connected.
We denote by $e_{G}:\mathrm{Spec}(k)\rightarrow G$ the neutral element
of $G$ and by $\mathrm{m}_{G}:G\times G\rightarrow G$ the group
law morphism. 
\begin{defn}
A \emph{rational action} of an algebraic group $G$ on a variety $X$
is a rational map $\alpha:G\times X\dashrightarrow X$ such that the
following diagrams of rational maps commute

\begin{eqnarray}
\xymatrix@=4.5em{ G\times G\times X \ar@{-->}[r]^{\mathrm{id}_G\times \alpha} \ar[d]_{\mathrm{m}_G\times \mathrm{id}_X} & G\times X \ar@{-->}[d]^{\alpha} & & \mathrm{Spec}(k)\times X \ar[r]^{e_G\times \mathrm{id}_X} \ar[dr]_{\mathrm{pr}_2} & G\times X\ar@{-->}[d]^{\alpha} \\ G\times X \ar@{-->}[r]^{\alpha} & X & & & X.}
\label{eq:GeomDiagrams}
\end{eqnarray}We denote by $\mathrm{dom}(\alpha)$ the largest open subset of $G\times X$
on which $\alpha$ is defined and we say that $\alpha:G\times X\dashrightarrow X$
is defined at a point $(g,x)\in G\times X$ if the latter belongs
to $\mathrm{dom}(\alpha)$. If so, we denote $\alpha(g,x)$ simply
by $g\cdot x$. Remark that $\mathrm{dom}(\alpha)\cap(\{e_{G}\}\times X)$
is a non empty open subset of $\{e_{G}\}\times X$ \cite{De}. A rational
action $\alpha:G\times X\dashrightarrow X$ such that $\mathrm{dom}(\alpha)=G\times X$
is called \emph{regular}. 
\end{defn}
The conditions above mean equivalently that if $(g,x)$ and $(g',g\cdot x)$
belongs to $\mathrm{dom}(\alpha)$ then $(g'g,x)$ belongs to $\mathrm{dom}(\alpha)$
and $(g'g)\cdot x=g'\cdot(g\cdot x)$. Furthermore, if $(e_{G},x)\in\mathrm{dom}(\alpha)$
then $e_{G}\cdot x=x$. These can be rephrased more formally by saying
that rational actions of $G$ on $X$ correspond to homomorphisms
of group functors $G\rightarrow\mathrm{Bir}_{k}(X)$, where $\mathrm{Bir}_{k}(X)$
is the contravariant functor $(k\mbox{-}\mathrm{Varieties})\rightarrow(\mathrm{Groups})$
which associates to every $k$-variety $T$, the group of $T$-birational
maps $X\times T\dashrightarrow X\times T$. A rational action is regular
if and only the corresponding homomorphism $G\rightarrow\mathrm{Bir}_{k}(X)$
factors through the automorphism group functor $\mathrm{Aut}_{k}(X)$
of $X$.

\subsection{Criterion for existence of rational $\mathbb{G}_{a}$-actions}

A rational action $\alpha:\mathbb{G}_{a}\times X\dashrightarrow X$
of the additive group scheme $\mathbb{G}_{a}=\mathbb{G}_{a,k}=\mathrm{Spec}(k[t])$
on a $k$-variety $X$ with field of rational functions $K_{X}$ is
equivalently determined by a co-action homomorphism $\alpha^{*}:K_{X}\rightarrow K_{X}(t)=K_{X}\otimes_{k}k(t)$
of fields over $k$ which factors through the valuation ring $\mathcal{O}_{\nu_{0}}=\left\{ r(t)\in K_{X}(t)\mid\mathrm{ord}_{0}r(t)\geq0\right\} $
of $K_{X}(t)$ and such that the following diagrams commute \\
\begin{eqnarray} \xymatrix@=3.5em{ K_X \ar[r]^{\alpha^*} \ar[d]_{\alpha^*} & K_X(t) \ar[d]^{t\mapsto t+t'}   & K_X \ar[r]^-{\overline{\alpha^*}} \ar[dr]_{\mathrm{id}} & \mathcal{O}_{\nu_0}/t\mathcal{O}_{\nu_0} \\ K_X(t')=K_X\otimes_k k(t')\ar[r]^-{\alpha^*\otimes \mathrm{id}} & K_X(t)\otimes_k k(t')\simeq K_X(t,t')   & & K_X. \ar[u]}
\label{eq:AlgeDiagrams}
\end{eqnarray}
 Indeed, the condition that $\alpha^{*}$ factors through $\mathcal{O}_{\nu_{0}}$
is ensured by the fact that $\mathrm{dom}(\alpha)\cap(\{0\}\times X)$
is a non empty open subset of $\{0\}\times X$, and the commutativity
of the two diagrams expresses the usual axioms for a co-action. The
following characterization is well-known: 
\begin{prop}
\label{prop:CalculRusse} A $k$-variety $X$ admits a nontrivial
rational $\mathbb{G}_{a}$-action if and only if it is birationally
ruled, i.e., birationally isomorphic to $Y\times\mathbb{P}^{1}$ for
some $k$-variety $Y$. \end{prop}
\begin{proof}
Every $k$-variety of the form $Y\times\mathbb{P}^{1}$ admits a regular
$\mathbb{G}_{a}$-action by projective translation on the second factor.
The converse follows for instance from Rosenlicht Theorem \cite{Ro56}
which asserts for our purpose that a $k$-variety equipped with a
rational $\mathbb{G}_{a}$-action is $\mathbb{G}_{a}$-equivariantly
birationally isomorphic to $U\times\mathbb{G}_{a}$ on which $\mathbb{G}_{a}$
acts by translations on the second factor for some affine $k$-variety
$U$. Nevertheless we find more enlightening to give an elementary
proof borrowed from Koshevoi \cite{Ko67}. Suppose that $\alpha:\mathbb{G}_{a}\times X\dashrightarrow X$
is a nontrivial rational $\mathbb{G}_{a}$-action and let $K_{0}=K_{X}^{\mathbb{G}_{a}}=\left\{ h\in K_{X}\mid\alpha^{*}h=h\right\} $
be its field of invariants. It is enough to show that there exists
$s\in K_{X}\setminus K_{0}$ such that $\alpha^{*}s=s+t$ and $K_{X}=K_{0}(s)$.
Note that if such an element $s$ exists, then it is transcendental
over $K_{0}$ for otherwise, applying $\alpha^{*}$ to a nontrivial
polynomial relation $P(s)=0$ for some $P\in K_{0}[v]$ would render
the conclusion that $t\in K_{X}(t)$ is algebraic over $K_{0}(s)$
whence over $K_{X}$, which is absurd. Furthermore, since any two
elements $s_{i}$, $i=1,2$, such that $\alpha^{*}s_{i}=s_{i}+t$
differs only by the addition of an element in $K_{0}$, it is enough
to show that for every $f\in K_{X}\setminus K_{0}$ there exists $s\in K_{X}$
such that $\alpha^{*}s=s+t$ and $f\in K_{0}(s)$. 

Now since $\alpha$ is nontrivial, there exists $f\in K_{X}\setminus K_{0}$
and $\alpha^{*}f$ can be written in the form $\alpha^{*}(f)=\left(1+b(t)\right)^{-1}a(t)$
where $a(t)=\sum_{i=0}^{n}a_{i}t^{i}\in K[t]$ with $a_{0}=f$, $b(t)=\sum_{i=1}^{m}b_{i}t^{i}\in tK[t]$,
and either $a(t)$ or $1+b(t)$ is nonconstant. The commutativity
of the first diagram \ref{eq:AlgeDiagrams} above implies that 
\begin{eqnarray*}
\left(1+\sum_{i=1}^{m}\alpha^{*}(b_{i})(t')^{i}\right)^{-1}\left(\sum_{i=0}^{n}\alpha^{*}(a_{i})(t')^{i}\right) & = & \left(1+\sum_{i=1}^{m}b_{i}(t+t')^{i}\right)^{-1}\left(\sum_{i=0}^{n}a_{i}(t+t')^{i}\right)\\
 & = & \left(1+\sum_{i=1}^{m}b_{i}t^{i}+\sum_{i=1}^{m}b_{1,i}(t)(t')^{i}\right)^{-1}\left(\sum_{i=0}^{n}a_{1,i}(t)(t')^{i}\right)\\
 & = & \left(1+\sum_{i=1}^{m}\frac{b_{1,i}(t)}{1+\sum_{i=1}^{m}b_{i}t^{i}}(t')^{i}\right)^{-1}\left(\sum_{i=0}^{n}\frac{a_{1,i}(t)}{1+\sum b_{i}t^{i}}(t')^{i}\right)
\end{eqnarray*}
where $a_{1,i}(t)=\sum_{j=i}^{n}\binom{j}{j-i}a_{j}t^{j-i}$ and $b_{1,i}(t)=\sum_{j=i}^{m}\binom{j}{j-i}b_{j}t^{j-i}$.
Identifying the coefficients, we obtain 
\begin{eqnarray*}
\alpha^{*}(a_{j})=\frac{a_{1,j}(t)}{1+\sum_{i=1}^{m}b_{i}t^{i}} & \textrm{and } & \alpha^{*}(b_{j})=\frac{b_{1,j}(t)}{1+\sum_{i=1}^{m}b_{i}t^{i}}.
\end{eqnarray*}
In particular, $\alpha^{*}(a_{n}^{-1})=(a_{n}^{-1}+\sum_{i=1}^{m}a_{n}^{-1}b_{i}t^{i})\in K[t]$
and, re-using the axioms to get the equality 
\[
\alpha^{*}a_{n}^{-1}+\sum_{i=1}^{m}\alpha^{*}(a_{n}^{-1}b_{i})(t')^{i}=a_{n}^{-1}+\sum_{i=1}^{m}a_{n}^{-1}b_{i}(t+t')^{i},
\]
we deduce that $\alpha^{*}(a_{n}^{-1}b_{i})=a_{n}^{-1}\sum_{j=i}^{m}\binom{j}{i}b_{j}t^{j-i}$
for every $i=1,\ldots m$. Thus $a_{n}^{-1}b_{m}\in K_{0}$, $\alpha^{*}(a_{n}^{-1}b_{m-1})=a_{n}^{-1}b_{m-1}+ma_{n}^{-1}b_{m}t$
and so, letting $s=\frac{a_{n}^{-1}b_{m-1}}{ma_{n}^{-1}b_{m}}$ we
obtain that $\alpha^{*}s=s+t$. We further deduce by induction that
$a_{n}^{-1}b_{i}\in K_{0}[s]$ for every $i=1,\ldots,m$. The same
argument applied to $f^{-1}$ implies that $s'=\frac{a_{n-1}b_{m}^{-1}}{na_{n}b_{m}^{-1}}$
also satisfies $\alpha^{*}s'=s'+t$ and that $f^{-1}b_{m}^{-1}a_{i}\in K_{0}[s']=K_{0}[s]$
for every $i=1,\ldots,n$. Since $b_{m}^{-1}a_{n}\in K_{0}$, this
shows that $f\in K_{0}(s)$ as desired. 
\end{proof}
The proof above shows more precisely that for every nontrivial rational
$\mathbb{G}_{a}$-action $\alpha:\mathbb{G}_{a}\times X\dashrightarrow X$
there exists a decomposition $K_{X}=K_{X}^{\mathbb{G}_{a}}(s)$, where
$K_{X}^{\mathbb{G}_{a}}$ is the field of invariant and where $s$
is an element transcendental over $K_{X}^{\mathbb{G}_{a}}$ satisfying
$\alpha^{*}s=s+t$, for which $\alpha^{*}$ takes the form 
\begin{equation}
\alpha^{*}=\alpha_{(K_{X}^{\mathbb{G}_{a}},s)}^{*}:K_{X}=K_{X}^{\mathbb{G}_{a}}(s)\rightarrow K_{X}^{\mathbb{G}_{a}}(s)(t),\; f(s)\mapsto\alpha_{(K_{X}^{\mathbb{G}_{a}},s)}^{*}(f(s))=f(s+t).\label{eq:StandardForm}
\end{equation}
An element $s\in K_{X}$ with the above properties is called a \emph{rational
slice} for the action $\alpha$. 
\begin{example}
A smooth curve $C$ admits a rational $\mathbb{G}_{a,k}$-action if
and only it is birational to $\mathbb{P}_{K_{0}}^{1}$ for a certain
algebraic extension $K_{0}$ of $k$. Indeed, by Proposition \ref{prop:CalculRusse}
above, $C$ admits a rational $\mathbb{G}_{a,k}$-action if and only
if $K_{C}=K_{0}(s)$ for some element $s$ transcendental over $K_{0}$.
This implies that $K_{0}$ is an algebraic extension of $k$ and that
$C\stackrel{\sim}{\dashrightarrow}\mathbb{P}_{K_{0}}^{1}$.
\end{example}

\subsection{Rational $\mathbb{G}_{a}$-actions and rational vector fields\label{sub:Rat-vectorFields}}

Every rational $\mathbb{G}_{a}$-action $\alpha:\mathbb{G}_{a}\times X\dashrightarrow X$
on a $k$-variety $X$ gives rise to a \emph{rational vector field},
i.e. a $k$-derivation $\tilde{\partial}:\mathcal{O}_{X}\rightarrow\mathcal{K}_{X}$
from the structure sheaf $\mathcal{O}_{X}$ to the constant sheaf
$\mathcal{K}_{X}$ of rational functions on $X$, consisting of velocity
vectors along germs of general orbits. More precisely, $\alpha$ induces
a rational homomorphism of sheaves 
\[
\eta:\alpha^{*}\Omega_{X/k}^{1}\rightarrow\Omega_{\mathbb{G}_{a}\times X/k}^{1}\rightarrow\Omega_{\mathbb{G}_{a}\times X/X}^{1}
\]
on $\mathbb{G}_{a}\times X$, where $\Omega_{\mathbb{G}_{a}\times X/X}^{1}$
is the sheaf of relative differentials of the second projection $\mathrm{pr}_{X}:\mathbb{G}_{a}\times X\rightarrow X$.
Pulling back by the zero section morphism $e_{X}:X\rightarrow\mathbb{G}_{a}\times X$,
$x\mapsto(0,x)$, whose image intersects $\mathrm{dom}(\alpha)$ by
definition, we obtain a global section $e_{X}^{*}\eta:e_{X}^{*}\alpha^{*}\Omega_{X/k}^{1}\simeq\Omega_{X/k}^{1}\rightarrow e_{X}^{*}\Omega_{\mathbb{G}_{a}\times X/X}^{1}\simeq\mathcal{O}_{X}$
of the sheaf $\mathcal{H}om_{X}(\Omega_{X/k}^{1},\mathcal{O}_{X})\otimes\mathcal{K}_{X}$,
hence by composition with the canonical $k$-derivation $d:\mathcal{O}_{X}\rightarrow\Omega_{X/k}^{1}$,
a $k$-derivation $\tilde{\partial}:\mathcal{O}_{X}\rightarrow\mathcal{K}_{X}$.
Furthermore, we can extend this derivation via the Leibniz rule to
a $k$-derivation from $\mathcal{K}_{X}$ to $\mathcal{K}_{X}$. We
denote this derivation with the same symbol $\tilde{\partial}:\mathcal{K}_{X}\rightarrow\mathcal{K}_{X}$. 

If the $\mathbb{G}_{a}$-action $\alpha$ is regular, then $\eta:\alpha^{*}\Omega_{X/k}^{1}\rightarrow\Omega_{\mathbb{G}_{a}\times X/X}^{1}$
is regular homomorphism, giving rise to global section $e_{X}^{*}\eta$
of $\mathcal{H}om_{X}(\Omega_{X/k}^{1},\mathcal{O}_{X})$, for which
the corresponding derivation $\tilde{\partial}:\mathcal{O}_{X}\rightarrow\mathcal{K}_{X}$
factors through $\mathcal{O}_{X}$. In the case of a regular $\mathbb{G}_{a}$-action
$\alpha:\mathbb{G}_{a}\times X\rightarrow X$ on a affine variety
$X=\mathrm{Spec}(A)$, the $k$-derivation $\partial=\Gamma(X,\tilde{\partial})\in\mathrm{Der}_{k}(A)$
deduced from $\tilde{\partial}:\mathcal{O}_{X}\rightarrow\mathcal{O}_{X}$
coincides simply with the one ${\displaystyle \partial=}\frac{d}{dt}\mid_{t=0}\circ\alpha^{*}:A\rightarrow A[t]/tA[t]\simeq A$.
It is well-know (see e.g. \cite{MLNotes}) that a $k$-derivation
$\partial\in\mathrm{Der}_{k}(A)$ arises from a regular $\mathbb{G}_{a}$-action
on $X$ if and only if it is ``algebraically integrable'' in the
sense that the formal exponential homomorphism $\exp(t\partial):A\rightarrow A[[t]]$
factors through a homomorphism $\alpha^{*}:A\rightarrow A[t]\subset A[[t]]$.
This holds precisely when $A=\bigcup_{n\geq1}\mathrm{Ker}\partial^{n}$,
and derivations with this property are called \emph{locally nilpotent}. 

Being locally nilpotent is not a local property in the Zariski topology
since for instance the restriction of a locally nilpotent derivation
to a non $\mathbb{G}_{a}$-stable affine open subset of $X$ is no
longer locally nilpotent (see example \ref{ex:bad-restriction-1}
below). In contrast, the following weaker form of the algebraic integrability
condition behaves well under localization:
\begin{defn}
\label{def:rat-int}A $k$-derivation $\tilde{\partial}:\mathcal{K}_{X}\rightarrow\mathcal{K}_{X}$
on a variety $X$ is called \emph{rationally integrable }if the formal
exponential homomorphism 
\[
\exp(t\tilde{\partial}):\mathcal{K}_{X}\rightarrow\mathcal{K}_{X}[[t]],\quad f\mapsto\sum\frac{\tilde{\partial}^{n}f}{n!}t^{n}
\]
factors through $\mathcal{K}_{X}(t)\cap\mathcal{K}_{X}[[t]]$. 
\end{defn}
By definition, every rationally integrable $k$-derivation $\tilde{\partial}:\mathcal{K}_{X}\rightarrow\mathcal{K}_{X}$
induces a global rational $k$-derivation $\partial=\Gamma(X,\tilde{\partial}):K_{X}\rightarrow K_{X}$
which gives rise in turn to a homomorphism $\alpha^{*}=\exp(t\partial):K_{X}\rightarrow K_{X}(t)$
factoring through $\mathcal{O}_{\nu_{0}}$ and satisfying the axioms
of a rational co-action of $\mathbb{G}_{a}$. Conversely, for every
rational $\mathbb{G}_{a}$-action $\alpha:\mathbb{G}_{a}\times X\dashrightarrow X$
with associated co-morphism $\alpha^{*}:K_{X}\rightarrow K_{X}(t)$,
the fact that $\alpha^{*}$ factors through $\mathcal{O}_{\nu_{0}}$
guarantees that the $k$-linear homomorphism 
\begin{equation}
\partial=\overline{\frac{d}{dt}\circ\alpha^{*}}:K_{X}\rightarrow\mathcal{O}_{\nu_{0}}\stackrel{\frac{d}{dt}}{\rightarrow}\mathcal{O}_{\nu_{0}}\rightarrow\mathcal{O}_{\nu_{0}}/t\mathcal{O}_{\nu_{0}}\simeq K_{X}\label{eq:def-derivation}
\end{equation}
is well-defined and the commutativity of the second diagram \ref{eq:AlgeDiagrams}
above implies that $\partial$ is a $k$-derivation. In fact, if we
write $K_{X}=K_{X}^{\mathbb{G}_{a}}(s)$ for a suitable rational slice
$s$ in such a way that $\alpha^{*}$ takes the form $\alpha_{(K_{X}^{\mathbb{G}_{a}},s)}^{*}$
as in (\ref{eq:StandardForm}) above, then $\partial$ coincides with
the $k$-derivation $\frac{\partial}{\partial s}:K_{X}^{\mathbb{G}_{a}}(s)\rightarrow K_{X}^{\mathbb{G}_{a}}(s)$.
We deduce in turn from Taylor's formula that 
\[
\exp(t\partial)(f(s))=\sum\frac{t^{n}}{n!}\frac{\partial^{n}}{\partial s^{n}}f(s)=f(s+t)=\alpha^{*}(f(s)).
\]
Summing up, we obtain the following characterization: 
\begin{thm}
\label{prop:RatIntegrable} There exists a one-to-one correspondence
between rational $\mathbb{G}_{a}$-actions $\alpha:\mathbb{G}_{a}\times X\dashrightarrow X$
on a $k$-variety $X$ and rationally integrable $k$-derivations
$\tilde{\partial}:\mathcal{K}_{X}\rightarrow\mathcal{K}_{X}$. 
\end{thm}
For a rational $\mathbb{G}_{a}$-action $\alpha:\mathbb{G}_{a}\times X\dashrightarrow X$
associated with a rationally integrable $k$-derivation $\partial=\Gamma(X,\tilde{\partial}):K_{X}\rightarrow K_{X}$,
the field of invariants $K_{X}^{\mathbb{G}_{a}}$ is equal to the
kernel $\mathrm{Ker}\partial$ of $\partial$ while rational slices
for $\alpha$ coincides precisely with elements $s\in K_{X}$ such
that $\partial s=1$. 
\begin{rem}
In \cite{MLNotes}, a $k$-derivation $\partial:K\rightarrow K$ of
a field extension $K/k$ is called locally nilpotent if $K$ is equal
to the field of fractions of its sub-ring $\mathrm{Nil}(\partial)=\bigcup_{n\geq0}\mathrm{Ker}\partial^{n}$.
In the case where $K=K_{X}$ is the field of rational functions on
a $k$-variety $X$, this property turns out to be equivalent to the
rational integrability of the associated derivation $\partial:\mathcal{K}_{X}\rightarrow\mathcal{K}_{X}$.
Indeed, by virtue of \cite[Lemma 2 p. 13]{MLNotes} and Proposition
\ref{prop:CalculRusse} above the two notions are both equivalent
to the property that $K_{X}$ is a purely transcendental extension
of its subfield $\mathrm{Ker}\partial$. The formulation in terms
of rational integrability has the advantage to be easier to check
in practice: by definition, if $K_{X}=k(f_{1},\ldots,f_{n})$ then
a $k$-derivation $\partial:K_{X}\rightarrow K_{X}$ is rationally
integrable if and only if $\exp(t\partial)(f_{i})\in K_{X}(t)$ for
every $i=1,\ldots,n$. \end{rem}
\begin{example}
\label{ex:bad-restriction-1}Let $\tilde{\partial}:\mathcal{O}_{\mathbb{A}^{1}}\rightarrow\mathcal{O}_{\mathbb{A}^{1}}$
be the $k$-derivation associated with the regular action of $\mathbb{G}_{a}$
on $\mathbb{A}^{1}=\mathrm{Spec}(k[x])$ by translations. Then $\Gamma(\mathbb{A}^{1},\tilde{\partial})=\frac{\partial}{\partial x}$
is a locally nilpotent derivation of $k[x]$. On the other hand, for
every non constant polynomial $p\in k[x]$, the $k$-derivation of
$k[x]_{p(x)}$ induced by $\tilde{\partial}$ is rationally integrable
but not locally nilpotent, defining a rational $\mathbb{G}_{a}$-action
of the principal open subset $U_{p}=\mathrm{Spec}(k[x]_{p(x)})$ of
$\mathbb{A}^{1}$. 
\end{example}
\noindent
\begin{example}
\label{ex:Bad-restriction-2} The derivation $\partial=-x^{2}\frac{\partial}{\partial x}:k[x]\rightarrow k[x]$
is not locally nilpotent. However, the equality 
\[
\exp(t\partial)(x)=\sum_{n=0}^{\infty}\frac{\partial^{n}x}{n!}t^{n}=\sum_{n=0}^{\infty}(-1)^{n}x^{n+1}t^{n}=\frac{x}{1+tx}
\]
in $k(t)[[x]]$ implies that the induced derivation of $k(x)$ is
rationally integrable with $s=x^{-1}$ as a slice, and hence defines
a rational $\mathbb{G}_{a}$-action $\alpha:\mathbb{G}_{a}\times\mathbb{A}^{1}\dashrightarrow\mathbb{A}^{1}$
on $\mathbb{A}^{1}=\mathrm{Spec}(k[x])$. In fact, $\alpha$ coincides
simply with the restriction to the open subset $\mathbb{P}^{1}\setminus\{\left[1:0\right]\}$
of $\mathbb{P}^{1}=\mathrm{Proj}(k[u,v])$ of the regular $\mathbb{G}_{a}$-action
$t\cdot[u:v]=[u:v+tu]$.
\end{example}
In the examples above, the derivation $\tilde{\partial}:\mathcal{O}_{X}\rightarrow\mathcal{K}_{X}$
factors through $\mathcal{O}_{X}$, in other words, the a priori rational
vector field is in fact regular. The following examples illustrate
the situation where the $\mathbb{G}_{a}$-action is induced by genuinely
rational vector fields. 
\begin{example}
The $k$-derivation $\partial=x^{-1}\frac{\partial}{\partial y}:k(x,y)\rightarrow k(x,y)$
is rationally integrable and its associated rational $\mathbb{G}_{a}$-action
$\alpha:\mathbb{G}_{a}\times X\dashrightarrow X$, $(x,y)\mapsto\left(x,y+\frac{t}{x}\right)$
restricts to a regular one on the open subset $U=X_{x}=\mathrm{Spec}(k[x^{\pm1},y])$
where $\partial$ is actually locally nilpotent. But $\mathrm{dom}(\alpha)\cap(\{0\}\times X)=\{0\}\times U$
and in fact, $(t,p)\notin\mathrm{dom}(\alpha)$ for all $p\in X\setminus U$
and $t\in\mathbb{G}_{a}$.
\end{example}
\noindent
\begin{example}
\label{ex:affine-plane} Let $X=\mathbb{A}^{2}=\mathrm{Spec}(k[x,y])$.
By virtue of Proposition \ref{prop:CalculRusse} and Theorem \ref{prop:RatIntegrable},
a $k$-derivation $\partial:k(x,y)\rightarrow k(x,y)$ is rationally
integrable if and only if there exists an element $y_{0}\in k(x,y)$
purely transcendental over $K_{0}=\mathrm{Ker}\partial$ such that
$\partial(y_{0})=1$ and an isomorphism $k(x,y)\simeq K_{0}(y_{0})$.
By Luröth theorem, $K_{0}$ is itself purely transcendental over $k$,
say $K_{0}=k(x_{0})$ for some $x_{0}\in k(x,y)$. In other words,
we obtain the rational counterpart of a classical result of Rentschler
\cite{Re68} which asserts that up to a biregular coordinate change
on $\mathbb{A}^{2}$, every locally nilpotent $k$-derivation of $k[x,y]$
has the form $\partial=r(x)\frac{\partial}{\partial y}$ for some
polynomial $r(x)\in k[x]$, namely: up to a birational coordinate
change on $\mathbb{A}^{2}$, i.e. a $k$-automorphism of $k(x,y)$,
every rationally integrable $k$-derivation takes the form $\partial=r(x)\frac{\partial}{\partial y}$
for some some rational function $r(x)\in k(x)$. 
\end{example}

\section{Regular actions of the additive group on semi-affine varieties\label{sec:2}}

Recall that a $k$-variety $X$ is called semi-affine if the canonical
morphism $p:X\rightarrow X_{0}=\mathrm{Spec}(\Gamma(X,\mathcal{O}_{X}))$
is proper. In this case $\Gamma(X,\mathcal{O}_{X})$ is finitely generated
and so $X_{0}$ is an affine variety \cite[Corollary~3.6]{GL}. For
instance, complete or affine $k$-varieties are semi-affine. By the
previous subsection, every regular $\mathbb{G}_{a}$-action $\alpha:\mathbb{G}_{a}\times X\rightarrow X$
on a $k$-variety $X$ gives rise to a rationally integrable $k$-derivation
$\tilde{\partial}:\mathcal{O}_{X}\rightarrow\mathcal{O}_{X}$. Conversely,
the following theorem shows that in the case where $X$ is semi-affine,
a rationally integrable derivation $\tilde{\partial}:\mathcal{O}_{X}\rightarrow\mathcal{O}_{X}$
corresponds to a regular $\mathbb{G}_{a}$-action if and only if the
associated global $k$-derivation $\Gamma(X,\tilde{\partial}):\Gamma(X,\mathcal{O}_{X})\rightarrow\Gamma(X,\mathcal{O}_{X})$
is locally nilpotent. 
\begin{thm}
\label{thm:Regular-actions}Regular $\mathbb{G}_{a}$-actions on a
semi-affine variety $X$ are in one-to-one correspondence with rationally
integrable $k$-derivations $\tilde{\partial}:\mathcal{O}_{X}\rightarrow\mathcal{O}_{X}$
such that the derivation $\Gamma(X,\tilde{\partial}):\Gamma(X,\mathcal{O}_{X})\rightarrow\Gamma(X,\mathcal{O}_{X})$
on the ring of global regular functions is locally nilpotent.\end{thm}
\begin{proof}
By Rosenlicht theorem \cite{Ro56}, for any regular $\mathbb{G}_{a}$-action
on $X$ there exists of a nonempty $\mathbb{G}_{a}$-invariant affine
open subset $U$. Hence, $\Gamma(U,\tilde{\partial})$ is locally
nilpotent and since $\Gamma(X,\mathcal{O}_{X})\subset\Gamma(U,\mathcal{O}_{X})$
it follows that $\Gamma(X,\tilde{\partial})$ is a locally nilpotent
derivation of $\Gamma(X,\mathcal{O}_{X})$. Conversely, let $\tilde{\partial}:\mathcal{O}_{X}\rightarrow\mathcal{O}_{X}$
be a derivation such that $\partial_{0}=\Gamma(X,\tilde{\partial}):\Gamma(X,\mathcal{O}_{X})\rightarrow\Gamma(X,\mathcal{O}_{X})$
is locally nilpotent. Then $\partial_{0}$ induces a possibly trivial
regular $\mathbb{G}_{a}$-action $\alpha_{0}:\mathbb{G}_{a}\times X_{0}\rightarrow X_{0}$
on $X_{0}=\mathrm{Spec}(\Gamma(X,\mathcal{O}_{X}))$ for which the
canonical morphism $p:X\rightarrow X_{0}$ is $\mathbb{G}_{a}$-equivariant.
In particular, for every point $x\in X$, letting $\xi=\alpha\mid_{\mathbb{G}_{a}\times\left\{ x\right\} }:\mathbb{G}_{a}\dashrightarrow X$,
$t\mapsto\alpha(t,x)$ and $\xi_{0}=\alpha_{0}\mid_{\mathbb{G}_{a}\times p(x)}:\mathbb{G}_{a}\rightarrow X_{0},$$t\mapsto\alpha_{0}(t,p(x))$
, we have a commutative diagram \[\xymatrix {\mathbb{G}_a \ar@{-->}[r]^{\xi} \ar[dr]_{\xi_0} & X \ar[d]^{p} \\ & X_0.}\]
Since $p$ is proper, we deduce from the valuative criterion for properness
applied to the local ring of every closed point $t\in\mathbb{G}_{a}$
that $\alpha$ is defined at every point $\left(x,t\right)\in\mathbb{G}_{a}\times X$
whence is a regular $\mathbb{G}_{a}$-action on $X$. 
\end{proof}
As a consequence of the proof of the above Theorem, we obtain the
following criterion to decide whether a derivation gives rise to a
regular $\mathbb{G}_{a}$-action on a semi-affine variety: 
\begin{cor}
\label{cor:open-LND}Let $X$ be a semi-affine variety and let $\tilde{\partial}:\mathcal{O}_{X}\rightarrow\mathcal{O}_{X}$
be a $k$-derivation. Then $\tilde{\partial}$ defines a regular $\mathbb{G}_{a}$-action
on $X$ if and only if there exists a non empty affine open subset
$U\subset X$ such that $\Gamma(U,\tilde{\partial}):\Gamma(U,\mathcal{O}_{X})\rightarrow\Gamma(U,\mathcal{O}_{X})$
is locally nilpotent.\end{cor}
\begin{example}
The semi-affineness hypothesis cannot be weakened. For instance, letting
$X=\mathbb{A}_{*}^{2}=\mathrm{Spec}(k[x,y])\setminus\left\{ (0,0)\right\} $,
the derivation $\tilde{\partial}=\frac{\partial}{\partial x}:\mathcal{O}_{X}\rightarrow\mathcal{O}_{X}$
only defines a rational $\mathbb{G}_{a}$-action $\alpha:\mathbb{G}_{a}\times X\dashrightarrow X$
since for a point of the form $p=(x_{0},0)\in X$ the orbit map $\xi:\mathbb{G}_{a}\dashrightarrow X$,
$t\mapsto\alpha(t,p)=(x_{0}+t,0)$ is not defined at $t_{0}=-x_{0}$.
On the other hand, the restriction of $\frac{\partial}{\partial x}$
to the principal affine open subset $\left\{ y\neq0\right\} $ of
$X$ is locally nilpotent.
\end{example}
The previous example illustrates the typical situation where a rationally
integrable $k$-derivation $\tilde{\partial}:\mathcal{O}_{X}\rightarrow\mathcal{K}_{X}$
factoring through $\mathcal{O}_{X}$ does not give rise to a regular
$\mathbb{G}_{a}$-action $\alpha:\mathbb{G}_{a}\times X\rightarrow X$.
Namely, even though $\{0\}\times X$ is contained in the domain of
definition $\mathrm{dom}(\alpha)$ of $\alpha$, the $\mathbb{G}_{a}$-orbit
of a point $x$ might not be defined for every time $t\in\mathbb{G}_{a}$.
Nevertheless, in such situations, the following result, which is consequence
of a general construction due to Zaitsev \cite[Theorem 4.12]{Za}
(see also \cite{CD}) shows that it is always possible to find a minimal
equivariant partial completion of $X$, on which the $\mathbb{G}_{a}$-action
extends to a regular one:
\begin{prop}
Let $X$ be an algebraic variety equipped with a rational $\mathbb{G}_{a}$-action
$\alpha:\mathbb{G}_{a}\times X\dashrightarrow X$ associated to a
rationally integrable $k$-derivation $\tilde{\partial}:\mathcal{O}_{X}\rightarrow\mathcal{O}_{X}$.
Then there exists an algebraic variety $\overline{X}$ equipped with
a regular $\mathbb{G}_{a}$-action $\overline{\alpha}:\mathbb{G}_{a}\times\overline{X}\rightarrow\overline{X}$
and a $\mathbb{G}_{a}$-equivariant open immersion $j:X\hookrightarrow\overline{X}$.
Furthermore, such a triple $(\overline{X},\overline{\alpha},j)$ with
the additional property that $\overline{X}\setminus X$ contains no
$\mathbb{G}_{a}$-orbits is unique up to equivalence. 
\end{prop}

\section{Applications \label{sec:3}}

\subsection{The Rational Makar-Limanov invariant}

By analogy with the usual Makar-Limanov invariant \cite{MLNotes}
of an affine $k$-variety $X=\mathrm{Spec}(A)$, which is defined
as the sub-algebra $\mathrm{ML}(A)$ of $A$ consisting of regular
functions on $X$ which are invariant under all regular $\mathbb{G}_{a}$-action
on $X$, it is natural to define the \emph{Rational Makar-Limanov
invariant} of a $k$-variety $X$ as the sub-field $\mathrm{RML}(X)$
of $K_{X}$ consisting of rational functions on $X$ which are invariant
under all rational $\mathbb{G}_{a}$-actions on $X$. Equivalently,
$\mathrm{RML}(X)$ is equal to the intersection in $K_{X}$ of the
kernels of all rationally integrable $k$-derivations of $K_{X}$.
The $\mathrm{RML}$ invariant of a $k$-rational variety is clearly
equal to $k$ while Proposition~\ref{prop:CalculRusse} shows in
particular that $\mathrm{RML}(X)=K_{X}$ if and only if $X$ is not
birationally ruled. The following proposition provides the rational
counter-part of a result due to Makar-Limanov \cite[Lemma 21]{MLNotes}
which asserts that if $A$ is a $k$-algebra such that $\mathrm{ML}(A)=A$
then $\mathrm{ML}(A[x])=A$. 
\begin{prop}
If $X$ is not birationnally ruled then the projection $\mathrm{pr}_{X}:X\times\mathbb{P}^{1}\rightarrow X$
is invariant under all rational $\mathbb{G}_{a}$-actions on $X\times\mathbb{P}^{1}$. \end{prop}
\begin{proof}
Let $K_{X\times\mathbb{P}^{1}}=K_{X}(u)$ where $u$ is transcendental
over $K_{X}$. By virtue of Proposition \ref{prop:CalculRusse} above,
a rational $\mathbb{G}_{a}$-action $\alpha:\mathbb{G}_{a}\times(X\times\mathbb{P}^{1})\dashrightarrow X\times\mathbb{P}^{1}$
on $X\times\mathbb{P}^{1}$ gives rise to a decomposition $K_{X\times\mathbb{P}^{1}}=K_{X\times\mathbb{P}^{1}}^{\mathbb{G}_{a}}(s)$
for a suitable rational slice $s$. Letting $\nu_{0}$ be the restriction
of the $u^{-1}$-adic valuation on $K_{X\times\mathbb{P}^{1}}$ to
the sub-field $K_{X\times\mathbb{P}^{1}}^{\mathbb{G}_{a}}$, it is
enough to show that $\nu_{0}(x)=0$ for every $x\in K_{X\times\mathbb{P}^{1}}^{\mathbb{G}_{a}}$.
Indeed, noting that the residue field of the $u^{-1}$-adic valuation
on $K_{X\times\mathbb{P}^{1}}$ is equal to $K_{X}$, this will imply
that $K_{X\times\mathbb{P}^{1}}^{\mathbb{G}_{a}}$ is contained in
$K_{X}$ whence is equal to it since these two fields have the same
transcendence degree over $k$ and are both algebraically closed in
$K_{X\times\mathbb{P}^{1}}$. So suppose on the contrary that there
exists $x\in K_{X\times\mathbb{P}^{1}}^{\mathbb{G}_{a}}$ transcendental
over $k$ with $\nu_{0}(x)\neq0$. Up to changing $x$ for its inverse
we may assume that $\nu_{0}(x)<0$. It follows that the transcendence
degree of the residue field $\kappa_{0}$ of $\nu_{0}$ over $k$
is strictly smaller than that of $K_{X\times\mathbb{P}^{1}}^{\mathbb{G}_{a}}$.
The Ruled Residue Theorem \cite{Oh83} then implies that $K_{X}$
is a simple transcendental extension of the algebraic closure of $\kappa_{0}$
in $K_{X}$, in contradiction with the hypothesis that $X$ is not
birationally ruled. \end{proof}
\begin{cor}
A $k$-variety admits two rational $\mathbb{G}_{a}$-actions $\alpha_{i}:\mathbb{G}_{a}\times X\dashrightarrow X$,
$i=1,2$, such that for a general $k$-rational point $x\in X$ the
rational orbits maps $\alpha_{i}\mid_{\mathbb{G}_{a}\times\left\{ x\right\} }:\mathbb{G}_{a}\dashrightarrow X$,
$t\mapsto\alpha_{i}(t,x)$ do not coincide if and only if it is birationally
isomorphic over $k$ to $Y\times\mathbb{P}^{2}$ for some $k$-variety
$Y$. 
\end{cor}
One could have expected more generally that if $X$ is not birationally
ruled then for every $n\geq1$ the projection $\mathrm{pr}_{X}:X\times\mathbb{P}^{n}\rightarrow X$
is invariant under all rational $\mathbb{G}_{a}$-actions on $X\times\mathbb{P}^{n}$.
But this is wrong, as shown by the following example derived from
a famous counter-example to the birational version of the Zariski
Cancellation Problem \cite{BC85}.
\begin{example}
The affine threefold $X\subset\mathbb{A}_{\mathbb{C}}^{4}=\mathrm{Spec}(\mathbb{C}[x,y,z,t])$
defined by the equation 
\[
y^{2}+(t^{4}+1)(t^{6}+t^{4}+1)z^{2}=2x^{3}+3t^{2}x^{2}++t^{4}+1
\]
has no nontrivial rational $\mathbb{G}_{a}$-actions but $\mathrm{RML}(X\times\mathbb{A}^{3})=\mathbb{C}$. \end{example}
\begin{proof}
By virtue of Exemple 2.9 in \cite{MB84}, $X$ is a unirational, non-rational
affine variety with the property that $X\times\mathbb{A}^{3}$ is
rational. So $\mathrm{RML}(X\times\mathbb{A}^{3})=\mathbb{C}$ and
it remains to check that $\mathrm{RML}(X)=K_{X}$. By virtue of Proposition~\ref{prop:CalculRusse},
the existence of a nontrivial rational $\mathbb{G}_{a}$-action on
$X$ would imply that $X$ is birationnally isomorphic to $S\times\mathbb{A}^{1}$
for a smooth affine surface $S$. But since $X$ is unirational, $S$
would be unirational whence rational and so would be $X$, a contradiction. \end{proof}
\begin{rem}
In the regular case, an example of a smooth rational affine surface
$S=\mathrm{Spec}(A)$ such that $\mathrm{ML}(S)=A$ but $\mathrm{ML}(S\times\mathbb{A}^{2})=\mathbb{C}$
was given in \cite{dub13}. 
\end{rem}

\subsection{Homogeneous rational $\mathbb{G}_{a}$-actions on toric varieties}

Recall that a \emph{toric variety} $X$ is a normal $k$-variety equipped
with an effective regular action $\mu:\mathbb{T}\times X\rightarrow X$
of a split torus $\mathbb{T}=\mathbb{G}_{m,k}^{n}$ having an open
orbit. A rational $\mathbb{G}_{a}$-action $\alpha:\mathbb{G}_{a}\times X\dashrightarrow X$
on $X$ is said to be $\mathbb{T}$-homogeneous if it semi-commutes
with the action of $\mathbb{T}$. This means equivalently that the
sub-group of $\mathrm{Bir}_{k}(X)$ generated by these actions is
isomorphic to an algebraic group of the form $\mathbb{T}\ltimes\mathbb{G}_{a}$.
In this subsection, we give a combinatorial characterization of homogeneous
rational $\mathbb{G}_{a}$-actions on a toric variety $X$ in terms
of their corresponding rationally integrable derivations.

Let us briefly recall from \cite{Fu} some basic facts about the combinatorial
description of toric varieties. Let $M=\mathrm{Hom}(\mathbb{T},\mathbb{G}_{m,k})$
be the character lattice and let $N=\mathrm{Hom}(\mathbb{G}_{m,k},\mathbb{T})$
be the 1-parameter subgroup lattice of $\mathbb{T}$. Following the
usual convention, we consider $M$ and $N$ as additive lattices and
we let $M_{\mathbb{Q}}=M\otimes_{\mathbb{Z}}\mathbb{Q}$ and $N_{\mathbb{Q}}=N\otimes_{\mathbb{Z}}\mathbb{Q}$.
A fan $\Sigma\in N_{\mathbb{Q}}$ is a finite collection of strongly
convex polyhedral cones such that every face of $\sigma\in\Sigma$
is contained in $\Sigma$ and for all $\sigma,\sigma'\in\Sigma$ the
intersection $\sigma\cap\sigma'$ is a face in both cones $\sigma$
and $\sigma'$. A toric variety $X_{\Sigma}$ is built from $\Sigma$
in the following way. For every $\sigma\in\Sigma$, we define an affine
toric variety $X_{\sigma}=\mathrm{Spec}(k[\sigma^{\vee}\cap M])$,
where $\sigma^{\vee}\subseteq M_{\mathbb{Q}}$ is the dual cone of
$\sigma$ and $k[\sigma^{\vee}\cap M]$ is the semigroup algebra of
$\sigma^{\vee}\cap M$, i.e., 
\[
k[\sigma^{\vee}\cap M]=\bigoplus_{m\in\sigma^{\vee}\cap M}k\cdot\chi^{m},\quad\mbox{with}\quad\chi^{0}=1,\mbox{ and }\chi^{m}\cdot\chi^{m'}=\chi^{m+m'},\ \forall m,m'\in\sigma^{\vee}\cap M\,.
\]
Furthermore, if $\tau\subseteq\sigma$ is a face of $\sigma$, then
the inclusion of algebras $k[\sigma^{\vee}\cap M]\hookrightarrow k[\tau^{\vee}\cap M]$
induces a $\mathbb{T}$-equivariant open embedding $X_{\tau}\hookrightarrow X_{\sigma}$.
The toric variety $X_{\Sigma}$ associated to the fan $\Sigma$ is
then defined as the variety obtained by gluing the family $\{X_{\sigma}\mid\sigma\in\Sigma\}$
along the open embeddings $X_{\sigma}\hookleftarrow X_{\sigma\cap\sigma'}\hookrightarrow X_{\sigma'}$
for all $\sigma,\sigma'\in\Sigma$.

Let $X_{\Sigma}$ be a toric variety. Since the torus $\mathbb{T}$
acts on $X_{\Sigma}$ with an open orbit, the field of fractions $K_{X}$
of $X$ is equal to $K_{\mathbb{T}}=\mathrm{Frac}(k[M])$ which is
a purely transcendental extension of $k$ of degree $n=\dim\mathbb{T}$.
Let $\alpha:\mathbb{G}_{a}\times X\dashrightarrow X$ be a rational
$\mathbb{T}$-homogeneous $\mathbb{G}_{a}$-action on $X$, let $\tilde{\partial}:\mathcal{K}_{\mathbb{T}}\rightarrow\mathcal{K}_{\mathbb{T}}$
be the corresponding rational $k$-derivation and let $\partial=\Gamma(\mathbb{T},\tilde{\partial}):K_{\mathbb{T}}\rightarrow K_{\mathbb{T}}$
be the induced $k$-derivation of $K_{\mathbb{T}}$. In the case where
$\alpha$ is regular, it is well known that $\alpha$ is $\mathbb{T}$-homogeneous
if and only if $\partial$ is homogeneous, i.e., homogeneous as a
linear map with respect to the $M$-grading on $k[M]$. In the rational
case, the field $K_{\mathbb{T}}$ is not graded but it is the fraction
field of the $M$-graded ring $k[M]$, so we say that $f\in K_{\mathbb{T}}$
is homogeneous if $f$ is a quotient of homogeneous elements. We say
that a derivation $\partial:K_{\mathbb{T}}\rightarrow K_{\mathbb{T}}$
is homogeneous if it sends homogeneous elements to homogeneous elements. 
\begin{lem}
A rational $\mathbb{G}_{a}$-action $\alpha:\mathbb{G}_{a}\times\mathbb{T}\dashrightarrow\mathbb{T}$
is $\mathbb{T}$-homogeneous if and only if the corresponding $k$-derivation
$\partial:k[M]\rightarrow K_{\mathbb{T}}$ is homogeneous. Furthermore,
every homogeneous rational $k$-derivation $\partial:k[M]\rightarrow K_{\mathbb{T}}$
is regular, i.e. factors through $k[M]$.\end{lem}
\begin{proof}
The first assertion follows from the same argument as in the regular
case, see e.g. \cite[Lemma~2]{Li}. Since every homogeneous element
in $k[M]$ is invertible, it follows that the only homogeneous elements
in $K_{\mathbb{T}}$ are the characters $\chi^{m}$, $m\in M$, which
are regular functions on $\mathbb{T}$.
\end{proof}
Regular homogeneous $k$-derivations on $\mathbb{T}$ were already
described in \cite{De}, see also \cite[Proposition~3.1]{KKL}. Let
$p\in N$ and let $e\in M$. The linear map $\partial_{p,e}:k[M]\rightarrow k[M]$,
$\chi^{m}\mapsto p(m)\chi^{m+e}$ is a homogeneous $k$-derivation
on $\mathbb{T}$ and every homogeneous $k$-derivation on $\mathbb{T}$
is a multiple $\partial_{p,e}$ for some $e\in M$ and some $p\in N$.
Without loss of generality we may assume that $p$ is primitive.
\begin{lem}
Let $p\in N$ be a primitive vector and let $e\in M$. The $k$-derivation
$\partial_{p,e}$ is rationally integrable if and only if $p(e)=\pm1$.\end{lem}
\begin{proof}
Since $\partial_{-p,e}=-\partial_{p,e}$, we may assume without loss
of generality that $p(e)\geq0$. Choosing mutually dual basis for
$M$ and $N$, we may assume $p=(1,0,\ldots,0)$ and $e=(e_{1},\ldots,e_{n})$
with $e_{1}\geq0$. Letting $x_{i}=\chi^{\beta_{i}}$, where $\{\beta_{1},\ldots,\beta_{n}\}$
is the basis for $M$, the $k$-derivation $\partial_{p,e}$ becomes
\[
\partial_{p,e}=x_{1}^{e_{1}+1}x_{2}^{e_{2}}\cdots x_{n}^{e_{n}}\frac{\partial}{\partial x_{1}}.
\]
A direct computation now shows that $\partial_{p,e}$ is rationally
integrable if and only if $e_{1}=p(e)=1$.
\end{proof}
The following lemma gives conditions for a derivation $\partial_{p,e}$
to extend to a regular $k$-derivation of an affine toric variety
$X_{\sigma}$. It was first proven in \cite{De} in a slightly different
form (see also \cite[Proposition~3.1]{KKL} for a modern proof). For
a fan $\Sigma$ or a cone $\sigma$ the notation $\Sigma(1)$ and
$\sigma(1)$ refers to the set of primitive vectors of the rays in
$\Sigma$ and $\sigma$, respectively. 
\begin{lem}
\label{lem:vect-toric}Let $X_{\sigma}$ be an affine toric variety.
Then the homogeneous $k$-derivation $\partial_{p,e}$ on $\mathbb{T}$
extends to a $k$-derivation on $X_{\sigma}$ if and only if 
\begin{enumerate}
\item $e\in\sigma_{M}^{\vee}$, or
\item \label{LND-toric} There exists $\rho_{e}\in\sigma(1)$ such that
$p=\pm\rho_{e}$, $\rho_{e}(e)=-1$, and $\rho(e)\geq0$ for all $\rho\in\sigma(1)\setminus\{\rho_{e}\}$.
 
\end{enumerate}
Furthermore, $\partial_{e,p}$ is locally nilpotent if and only if
it is as in \eqref{LND-toric}. 
\end{lem}
By the valuative criterion for properness, a toric variety $X_{\Sigma}$
is semi-affine if and only if $\mathrm{Supp}(\Sigma)=\bigcup_{\sigma\in\Sigma}\sigma$
is convex. We can now apply Corollary~\ref{cor:open-LND} to recover
a description of regular $\mathbb{G}_{a}$-actions on semi-affine
toric varieties which was obtained by Demazure \cite{De} using lengthly
explicit computations. 
\begin{prop}
\label{prop:toric}Let $X_{\Sigma}$ be a semi-affine toric variety.
Then $\partial_{p,e}$ is the derivation of a $\mathbb{T}$-homogeneous
regular $\mathbb{G}_{a}$-actions $\alpha_{p,e}:\mathbb{G}_{a}\times X_{\Sigma}\rightarrow X_{\Sigma}$
on $X_{\Sigma}$ if and only if there exists $\rho_{e}\in\Sigma(1)$
such that $p=\pm\rho_{e}$, $\rho_{e}(e)=-1$, and $\rho(e)\geq0$
for all $\rho\in\Sigma(1)\setminus\{\rho_{e}\}$.     \end{prop}
\begin{proof}
By Corollary~\ref{cor:open-LND}, the $k$-derivation $\partial_{p,e}$
is the derivation of a $\mathbb{T}$-homogeneous regular $\mathbb{G}_{a}$-action
if and only if there exists an affine open $\mathbb{G}_{a}$-invariant
subset $U\subseteq X_{\Sigma}$ such that $\Gamma(U,\tilde{\partial})$
is locally nilpotent. Since the action is $\mathbb{T}$-homogeneous,
we can assume that $U$ is also $\mathbb{T}$-invariant. Now the proposition
follows from Lemma~\ref{lem:vect-toric}.
\end{proof}

\subsection{Rational $\mathbb{G}_{a}$-actions associated with affine-linear
bundles of rank one}

Here we consider a class of rational $\mathbb{G}_{a}$-actions coming
from regular actions of certain non constant groups schemes, locally
isomorphic to $\mathbb{G}_{a}$. We characterize the simplest possible
ones in terms of their corresponding rationally integrable $k$-derivations. 

Let us first note that every line bundle $p:L\rightarrow Z$ over
a $k$-variety $Z$ carries a canonical rationally integrable $\mathcal{O}_{Z}$-derivation
$d_{L/Z}:\mathcal{O}_{L}\rightarrow\Omega_{L/Z}^{1}\hookrightarrow\mathcal{K}_{L}$
with the property that over every affine open subset $Z_{i}$ on which
$L$ becomes trivial, the $\Gamma(Z_{i},\mathcal{O}_{Z_{i}})$-derivation
\[
\Gamma(p^{-1}(Z_{i}),d_{L/Z}):\Gamma(p^{-1}(Z_{i}),\mathcal{O}_{L})\rightarrow\Gamma(p^{-1}(Z_{i}),\Omega_{L/X}^{1})\simeq\Gamma(p^{-1}(Z_{i}),\mathcal{O}_{L})
\]
is locally nilpotent. Indeed, writing $p:L=\mathrm{Spec}_{Z}(\mathrm{Sym}_{Z}\mathcal{L}^{\vee})\rightarrow Z$
for a certain invertible sheaf $\mathcal{L}$, we have $\Omega_{L/Z}^{1}\simeq p^{*}\mathcal{L}^{\vee}$
and for every affine open subset $Z_{i}$ of $Z$ over which $L$-becomes
trivial, say $L\mid_{Z_{i}}\simeq\mathrm{Spec}(\mathcal{O}_{Z_{i}}[s_{i}])$,
$\Gamma(p^{-1}(Z_{i}),d_{L/Z})$ coincides with the derivation $\frac{\partial}{\partial s_{i}}$.
\\

A line bundle is in fact a group scheme over $Z$, locally isomorphic
to $\mathbb{G}_{a,Z}=\mathbb{G}_{a}\times_{\mathrm{Spec}(k)}Z$, whose
group law $m:L\times_{Z}L\rightarrow L$ is induced by the diagonal
homomorphism $\mathcal{L}\rightarrow\mathcal{L}\oplus\mathcal{L}$
of the invertible sheaf $\mathcal{L}$ of germs of sections of $p:L\rightarrow Z$,
and whose neutral section $e:Z\rightarrow L$ corresponds to the zero
section of $\mathcal{L}$. In this context, the correspondence between
regular $\mathbb{G}_{a}$-actions of an affine variety $X=\mathrm{Spec}(A)$
and locally nilpotent $k$-derivation of $A$ extends to a correspondence
between regular actions $\mu:L\times_{Z}X\rightarrow X$ of $L$ on
a variety $q:X\rightarrow Z$ affine over $Z$ and ``locally nilpotent''
$\mathcal{O}_{Z}$-derivations $\tilde{\partial}:\mathcal{O}_{X}\rightarrow q^{*}\mathcal{L}^{\vee}$.
Namely, the derivation $\tilde{\partial}$ is the composition of the
canonical $\mathcal{O}_{Z}$-derivation $d_{X/Z}:\mathcal{O}_{X}\rightarrow\Omega_{X/Z}^{1}$
with the homomorphism of $\mathcal{O}_{X}$-module $\Omega_{X/Z}^{1}\rightarrow q^{*}\mathcal{L}^{\vee}$
obtained similarly as in subsection \ref{sub:Rat-vectorFields} above
by pulling back the homomorphism $\eta:\mu^{*}\Omega_{X/Z}^{1}\rightarrow\Omega_{L\times_{Z}X/X}^{1}\simeq\mathrm{pr}_{X}^{*}q^{*}\mathcal{L}^{\vee}$
of $\mathcal{O}_{L\times_{Z}X}$-module by the zero section morphism
$e\times\mathrm{id}_{X}:X\rightarrow L\times_{Z}X$. This derivation
is locally nilpotent in the sense that $q_{*}\mathcal{O}_{X}$ is
the union of the kernels of the $\mathcal{O}_{Z}$-linear homomorphisms
$\partial_{L,X}^{n}:q_{*}\mathcal{O}_{X}\rightarrow q_{*}\mathcal{O}_{L}\otimes_{\mathcal{O}_{Z}}(\mathcal{L}^{\vee})^{\otimes n}$,
$n\geq1$, defined inductively by $\partial_{L,X}^{1}=q_{*}\tilde{\partial}:q_{*}\mathcal{O}_{X}\rightarrow q_{*}q^{*}\mathcal{L}^{\vee}\simeq q_{*}\mathcal{O}_{X}\otimes_{\mathcal{O}_{Z}}\mathcal{L}^{\vee}$
and, for every $n\geq2$, as the composition $\partial_{L,Z}^{n}=(\partial_{L,Z}^{1}\otimes\mathrm{id})\circ\partial_{L,X}^{n-1}$
where 
\[
(\partial_{L,Z}^{1}\otimes\mathrm{id}):q_{*}\mathcal{O}_{X}\otimes_{\mathcal{O}_{Z}}(\mathcal{L}^{\vee})^{\otimes n-1}\rightarrow(q_{*}\mathcal{O}_{X}\otimes_{\mathcal{O}_{Z}}\mathcal{L}^{\vee})\otimes_{\mathcal{O}_{Z}}(\mathcal{L}^{\vee})^{\otimes n-1}\simeq q_{*}\mathcal{O}_{X}\otimes_{\mathcal{O}_{Z}}(\mathcal{L}^{\vee})^{\otimes n}.
\]
The action $\mu:L\times_{Z}X\rightarrow X$ is then recovered as the
morphism induced by the formal exponential homomorphism 
\[
\exp(t\partial_{L,X})=\sum_{n\geq0}\frac{\partial_{L,X}^{n}}{n!}t^{n}:q_{*}\mathcal{O}_{X}\rightarrow q_{*}\mathcal{O}_{X}\otimes_{\mathcal{O}_{Z}}(\bigoplus_{n\geq0}(\mathcal{L}^{\vee})^{\otimes n}t^{n})\simeq q_{*}\mathcal{O}_{X}\otimes_{\mathcal{O}_{Z}}\mathrm{Sym}_{Z}\mathcal{L}^{\vee}.
\]
\\

The simplest examples of varieties admitting an action of a line bundle
$p:L\rightarrow Z$ are principal homogeneous $L$-bundles, that is,
varieties $q:X\rightarrow Z$ equipped with an action of $L$ which
are locally equivariantly isomorphic over $Z$ to $L$ acting on itself
by translations. For such varieties, the corresponding rationally
integrable $\mathcal{O}_{Z}$-derivations $\tilde{\partial}:\mathcal{O}_{X}\rightarrow q^{*}\mathcal{L}^{\vee}$
have the additional property that there exists a covering of $Z$
by affine open subset subsets $Z_{i}\subset Z$ on which $\mathcal{L}$
becomes trivial and such that the induced derivation 
\[
\Gamma(q^{-1}(Z_{i}),\tilde{\partial}):\Gamma(q^{-1}(Z_{i}),\mathcal{O}_{X})\rightarrow\Gamma(q^{-1}(Z_{i}),q^{*}\mathcal{L}^{\vee})\simeq\Gamma(q^{-1}(Z_{i}),\mathcal{O}_{X})
\]
is locally nilpotent, with a regular slice $s_{i}\in\Gamma(q^{-1}(Z_{i}),\mathcal{O}_{X})$.
The following Proposition shows conversely that the existence on a
variety $X$ of a structure of principal homogeneous bundle under
a suitable line bundle $p:L\rightarrow Z$ can be decided, without
prior knowledge of $L$ and $Z$, from the consideration of certain
rationally integrable $k$-derivations $\tilde{\partial}:\mathcal{O}_{X}\rightarrow\mathcal{K}_{X}$. 
\begin{prop}
\label{prop:Torsor-derivation} Let $X$ be a $k$-variety and let
$\tilde{\partial}:\mathcal{O}_{X}\rightarrow\mathcal{N}$ be a rationally
integrable $k$-derivation with value in an invertible subsheaf $\mathcal{N}$
of $\mathcal{K}_{X}$. Suppose that there exists a covering of $X$
by affine open subsets $X_{i}$, $i\in I$, and trivializations $\psi_{i}:\mathcal{N}\mid_{X_{i}}\stackrel{\sim}{\rightarrow}\mathcal{O}_{X_{i}}$
such that the following holds 

a) For every $i\in I$, the $k$-derivation $\Gamma(X_{i},\psi_{i}\circ\tilde{\partial}):\Gamma(X_{i},\mathcal{O}_{X})\rightarrow\Gamma(X_{i},\mathcal{O}_{X})$
is locally nilpotent with a regular slice $s_{i}\in\Gamma(X_{i},\mathcal{O}_{X})$. 

b) For every $i,j\in I$, the invertible function $\psi_{i}\circ\psi_{j}^{-1}\mid_{X_{i}\cap X_{j}}\in\Gamma(X_{i}\cap X_{j},\mathcal{O}_{X}^{*})$
is contained in $\mathrm{Ker}(\Gamma(X_{i}\cap X_{j},\tilde{\partial}))$. 

Then there exists a geometrically integral scheme $Z$ of finite type
over $k$, a morphism $q:X\rightarrow Z$ and an invertible sheaf
$\mathcal{L}$ on $Z$ such that $\mathcal{N}\simeq q^{*}\mathcal{L}^{\vee}$
and $q:X\rightarrow Z$ is a principal homogeneous bundle under the
line bundle $p:\mathrm{Spec}_{Z}(\mathrm{Sym}_{Z}\mathcal{L}^{\vee})\rightarrow Z$. \end{prop}
\begin{proof}
Letting $\alpha_{i}:\mathbb{G}_{a}\times X_{i}\rightarrow X_{i}$
be the $\mathbb{G}_{a}$-action generated by the $k$-derivation $\partial_{i}=\Gamma(X_{i},\psi_{i}\circ\tilde{\partial})$
and $Z_{i}=\mathrm{Spec}(\Gamma(X_{i},\mathcal{O}_{X})/(s_{i}))\subset X_{i}$,
the first hypothesis implies that $\Phi_{i}:\mathbb{G}_{a}\times Z_{i}\rightarrow X_{i}$,
$(t,z_{i})\mapsto\alpha_{i}(t,z_{i})$ is a $\mathbb{G}_{a}$-equivariant
isomorphism between $\mathbb{G}_{a}\times Z_{i}$ equipped with the
action by translations on the first factor and $X_{i}$ equipped with
the action $\alpha_{i}$. By definition, $\partial_{i}\mid_{X_{i}\cap X_{j}}=a_{ij}\partial_{j}\mid_{X_{i}\cap X_{j}}$
where $a_{ij}=\psi_{i}\circ\psi_{j}^{-1}\mid_{X_{i}\cap X_{j}}\in\Gamma(X_{i}\cap X_{j},\mathcal{O}_{X}^{*})$
and condition b) says in particular that $a_{ij}\in\mathrm{Ker}\partial_{i}\mid_{X_{i}\cap X_{j}}=\mathrm{Ker}\partial_{j}\mid_{X_{i}\cap X_{j}}$.
This implies in turn that every element of $\Gamma(X_{i}\cap X_{j},\mathcal{O}_{X})$
which is in the canonical image of $\Gamma(X_{i},\mathcal{O}_{X})$
or $\Gamma(X_{j},\mathcal{O}_{X})$ is annihilated by a certain power
of $\partial_{i}$. Since $X$ is separated, $\Gamma(X_{i}\cap X_{j},\mathcal{O}_{X})$
is generated by these canonical images \cite[I.5.5.6]{EGA} and so
$\partial_{i}\mid_{X_{i}\cap X_{j}}$ and $\partial_{j}\mid_{X_{i}\cap X_{j}}$
are locally nilpotent derivations of $\Gamma(X_{i}\cap X_{j},\mathcal{O}_{X})$.
This shows that $X_{i}\cap X_{j}$ is stable under the $\mathbb{G}_{a}$-actions
$\alpha_{i}$ on $X_{i}$ and $\alpha_{j}$ on $X_{j}$. Therefore
there exists open subsets $Z_{ij}\simeq\mathrm{Spec}(\mathrm{Ker}\partial_{i}\mid_{X_{i}\cap X_{j}})$
and $Z_{ji}\simeq\mathrm{Spec}(\mathrm{Ker}\partial_{j}\mid_{X_{i}\cap X_{j}})$
of $Z_{i}$ and $Z_{j}$ respectively such that $X_{i}\cap X_{j}$
is simulnatenously $\mathbb{G}_{a}$-equivariantly isomorphic to $\mathrm{Spec}_{Z_{ij}}(\mathcal{O}_{Z_{ij}}[s_{i}])$
and $\mathrm{\mathrm{Spec}_{Z_{ji}}(\mathcal{O}_{Z_{ji}}[s_{j}])}$
with respect to the action $\alpha_{i}$ and $\alpha_{j}$. Furthermore,
since $a_{ij}\in\mathrm{Ker}\partial_{i}\mid_{X_{i}\cap X_{j}}$ we
have 
\[
\partial_{i}\mid_{X_{i}\cap X_{j}}(a_{ij}s_{i})=a_{ij}\partial_{i}\mid_{X_{i}\cap X_{j}}(s_{i})=a_{ij}=\partial_{i}\mid_{X_{i}\cap X_{j}}(s_{j})
\]
and so, there exists $b_{ij}\in\mathrm{Ker}\partial_{i}\mid_{X_{i}\cap X_{j}}=\mathrm{Ker}\partial_{j}\mid_{X_{i}\cap X_{j}}$
such that $s_{j}\mid_{X_{i}\cap X_{j}}=a_{ij}s_{i}\mid_{X_{i}\cap X_{j}}+b_{ij}$.
The same argument applied to a triple intersection $X_{i}\cap X_{j}\cap X_{k}$
shows that the natural isomorphisms $\varphi_{ij}:Z_{ji}\stackrel{\sim}{\rightarrow}Z_{ij}$
induced by the equality $\mathrm{Ker}\partial_{i}\mid_{X_{i}\cap X_{j}}=\mathrm{Ker}\partial_{j}\mid_{X_{i}\cap X_{j}}$
satisfy $\varphi_{jk}(Z_{ki}\cap Z_{kj})\subset Z_{jk}\cap Z_{ji}$
and $\varphi_{ik}\mid_{Z_{ki}\cap Z_{kj}}=\varphi_{ij}\mid_{Z_{jk}\cap Z_{ji}}\circ\varphi_{jk}\mid_{Z_{ki}\cap Z_{kj}}$.
This implies the existence of a unique $k$-scheme $Z$ together with
open immersions $\zeta_{i}:Z_{i}\hookrightarrow Z$ such that $\xi_{i}\circ\varphi_{ij}=\xi_{j}$.
Furthermore, the local projections $\mathrm{pr}_{Z_{i}}:X_{i}\simeq Z_{i}\times\mathbb{A}^{1}\rightarrow Z_{i}$
glue to a locally trivial $\mathbb{A}^{1}$-bundle $q:X\rightarrow Z$
with trivializations $\rho^{-1}(Z_{i})\simeq\mathrm{Spec}_{Z_{i}}(\mathcal{O}_{Z_{i}}[s_{i}])$,
$i\in I$, where we identified $Z_{i}$ with its image in $Z$. The
invertible functions $a_{ij}\in\Gamma(X_{i}\cap X_{j},\mathcal{O}_{X}^{*})\cap\mathrm{Ker}\partial_{i}\mid_{X_{i}\cap X_{j}}\simeq\Gamma(Z_{i}\cap Z_{j},\mathcal{O}_{Z}^{*})$
form a \v{C}ech $1$-cocycle with value in $\mathcal{O}_{Z}^{*}$
defining a unique invertible sheaf $\mathcal{L}^{\vee}$ such that
$\mathcal{N}\simeq q^{*}\mathcal{L}^{\vee}$, and the identity $s_{j}\mid_{X_{i}\cap X_{j}}=a_{ij}s_{i}\mid_{X_{i}\cap X_{j}}+b_{ij}$
says precisely that $q:X\rightarrow Z$ is in fact a principal homogeneous
bundle under the line bundle $p:\mathrm{Spec}_{Z}(\mathrm{Sym}_{Z}\mathcal{L^{\vee}})\rightarrow Z$
.\end{proof}
\begin{example}
Let $S$ be the smooth affine surface in $\mathbb{A}^{4}=\mathrm{Spec}(\mathbb{C}[x,y,z,u])$
defined by the equations 
\[
\begin{cases}
xz & =y(y-1)\\
yu & =z(z+1)\\
xu & =(y-1)(z+1)
\end{cases}
\]
and let $\partial,\partial':A=\Gamma(S,\mathcal{O}_{S})\rightarrow K_{S}$
be the $k$-derivations defined respectively by 
\begin{eqnarray*}
\begin{cases}
\partial x & =0\\
\partial y & =x^{2}\\
\partial z & =(2y-1)x\\
\partial u & =x(z+1)+(2y-1)(y-1)
\end{cases} & \textrm{and} & \begin{cases}
\partial'x & =\omega^{3}\\
\partial'y & =\omega^{2}\\
\partial'z & =\omega\\
\partial'u & =1
\end{cases}
\end{eqnarray*}
where $\omega=x/(y-1)\in K_{S}$. 

It is straightforward to check that $\partial$ is a locally nilpotent
$\mathbb{C}[x]$-derivation of $A$, thus defining a regular $\mathbb{G}_{a}$-action
$\alpha:\mathbb{G}_{a}\times S\rightarrow S$. The surface $S$ is
covered by the two $\mathbb{G}_{a}$-invariant affine open subsets
\begin{eqnarray*}
S_{0}=S\setminus\left\{ x=y-1=0\right\} \simeq\mathrm{Spec}(\mathbb{C}[x,v_{0}]) & \textrm{ and } & S_{1}=S\setminus\left\{ x=y=z+1=0\right\} \simeq\mathrm{Spec}(\mathbb{C}[x,v_{1}])
\end{eqnarray*}
where $v_{0}$ and $v_{1}$ denote the restriction to $S_{0}$ of
the rational functions $(y-x)/x^{2}$ and $\omega^{-1}$. The restrictions
of $\partial$ to $S_{0}$ and $S_{1}$ coincide respectively the
locally nilpotent derivations $\frac{\partial}{\partial v_{0}}$ and
$x\frac{\partial}{\partial v_{1}}$. So letting $C_{1}\subset S$
be the curve $\left\{ x=y-1=0\right\} $, we see that the derivation
of $\mathcal{O}_{S}$ into itself associated to $\partial$ factors
through a derivation $\tilde{\partial}:\mathcal{O}_{S}\rightarrow\mathcal{N}=\mathcal{O}_{S}(-C_{1})$.
By definition, $\mathcal{N}\mid_{S_{0}}=\mathcal{O}_{S_{0}}$, $\mathcal{N}\mid_{S_{1}}=x\mathcal{O}_{S_{1}}$
and using the isomorphisms $\psi_{0}=\mathrm{id}_{\mathcal{O}_{S_{0}}}$
and $\psi_{1}:x\mathcal{O}_{S_{1}}\rightarrow\mathcal{O}_{S_{1}}$,
$x\mapsto1$, we obtain that the two derivations $\partial_{0}=\Gamma(S_{0},\psi_{0}\circ\tilde{\partial})=\frac{\partial}{\partial v_{0}}$
and $\partial_{1}=\Gamma(S_{1},\psi_{1}\circ\tilde{\partial})=\frac{\partial}{\partial v_{1}}$
are locally nilpotent with respective slices $s_{0}=v_{0}$ and $s_{1}=v_{1}$
and respective geometric quotients $S_{0}/\mathbb{G}_{a}=S_{1}/\mathbb{G}_{a}=\mathrm{Spec}(\mathbb{C}[x])$.
Since $x^{-1}\in\Gamma(S_{0}\cap S_{1},\mathcal{O}_{S}^{*})=\mathbb{C}[x^{\pm1}]$
belongs to $\mathrm{Ker}(\Gamma(S_{0}\cap S_{1},\tilde{\partial}))$,
the hypothesis of Proposition \ref{prop:Torsor-derivation} are satisfied.
In this example, the corresponding scheme $Z$ is isomorphic to the
affine line with a double origin, obtained by gluing $S_{0}/\mathbb{G}_{a}$
and $S_{1}/\mathbb{G}_{a}$ by the identity outside their respective
origins $o_{0}$ and $o_{1}$, and $\mathcal{L}^{\vee}=\mathcal{O}_{Z}(-o_{1})$.
The initial $\mathbb{G}_{a}$-action defined by $\partial$ is recovered
from the action $\mu:L\times_{Z}S\rightarrow S$ of $L=\mathrm{Spec}_{Z}(\mathrm{Sym}_{Z}\mathcal{L})\rightarrow Z$
as the composition 
\[
\alpha=\mu\circ(\sigma\times\mathrm{id}_{S}):\mathbb{G}_{a}\times S\simeq\mathbb{G}_{a,Z}\times_{Z}S\rightarrow L\times_{Z}S\rightarrow S
\]
where $\sigma:\mathbb{G}_{a,Z}=\mathbb{G}_{a}\times_{\mathrm{Spec}(\mathbb{C})}Z=\mathrm{Spec}_{Z}(\mathcal{O}_{Z}[t])\rightarrow L=\mathrm{Spec}_{Z}(\mathrm{Sym}_{Z}\mathcal{L}^{\vee})$
is the group scheme homomorphism induced by the canonical global section
$\sigma$ of $\mathcal{O}_{Z}(o_{1})=\mathcal{H}om_{Z}(\mathcal{L}^{\vee},\mathcal{O}_{Z})$
with divisor equal to $o_{1}$. 

The second derivation $\partial'$ is not locally nilpotent. However,
noting that $\partial'\omega=0$ and that the restriction of $\partial'$
to the open subset $S_{1}\simeq\mathrm{Spec}(\mathbb{C}[x,v_{1}])$
coincides with the derivation $v_{1}^{-3}\frac{\partial}{\partial x}=\omega^{3}\frac{\partial}{\partial x}$,
we conclude that the associated derivation $\tilde{\partial}:\mathcal{O}_{S}\rightarrow\mathcal{K}_{S}$
is rationally integrable. Furthermore $\partial'$ restricts on the
open subset $S_{0}'=S\setminus\left\{ y-1=z=u=0\right\} \simeq\mathrm{Spec}(\mathbb{C}[u,v_{0}'])$,
where $v_{0}'=\omega\mid_{S_{0}'}$ to the locally nilpotent derivation
$\frac{\partial}{\partial u}$. The open subsets $S_{0}'$ and $S_{1}$
cover $S$ and letting $C_{0}'\subset S$ be the curve $\left\{ y-1=z=u=0\right\} $,
we see that $\tilde{\partial}$ factors through the invertible subsheaf
$\mathcal{N}'=\mathcal{O}_{S}(3C_{0}')$ of $\mathcal{K}_{S}$. By
definition, $\mathcal{N}'\mid_{S_{0}'}=\mathcal{O}_{S_{0}'}$, $\mathcal{N}'\mid_{S_{1}}=\omega^{-3}\mathcal{O}_{S_{1}}$
and using the isomorphisms $\psi_{0}'=\mathrm{id}_{\mathcal{O}_{S_{0}'}}$
and $\psi_{1}':\omega^{-3}\mathcal{O}_{S_{1}}\rightarrow\mathcal{O}_{S_{1}}$,
$\omega^{-3}\mapsto1$, we obtain that the two derivations $\partial_{0}'=\Gamma(S_{0}',\psi_{0}'\circ\tilde{\partial}')=\frac{\partial}{\partial u}$
and $\partial_{1}'=\Gamma(S_{1},\psi_{1}'\circ\tilde{\partial}')=\frac{\partial}{\partial x}$
are locally nilpotent with respective slices $s_{0}'=u$ and $s_{1}'=x$,
and respective geometric quotients $S_{0}'/\mathbb{G}_{a}=\mathrm{Spec}(\mathbb{C}[v_{0}'])$
and $S_{1}/\mathbb{G}_{a}=\mathrm{Spec}(\mathbb{C}[v_{1}])$. Since
$\omega^{-3}\in\Gamma(S_{0}'\cap S_{1},\mathcal{O}_{S}^{*})=\mathbb{C}[\omega^{\pm1}]$
belongs to $\mathrm{Ker}(\Gamma(S_{0}'\cap S_{1},\tilde{\partial}'))$,
the hypothesis of Proposition \ref{prop:Torsor-derivation} are again
satisfied. Here the corresponding scheme $Z$ is isomorphic to $\mathbb{P}^{1}$
obtained by gluing $S_{0}'/\mathbb{G}_{a}$ and $S_{1}/\mathbb{G}_{a}$
outside their respective origins $o_{0}'$ and $o_{1}$ by the isomorphism
$v_{0}'\mapsto v_{1}^{-1}$, and $\mathcal{L}^{\vee}\simeq\mathcal{O}_{Z}(3o_{0}')$.
The resulting morphism $q:S\rightarrow Z\simeq\mathbb{P}^{1}$, which
coincides with the one $(x,y,z,u)\mapsto[x:y-1]$, is thus a principal
homogeneous bundle under the geometric line bundle $L=\mathcal{O}_{\mathbb{P}^{1}}(-3)$
on $\mathbb{P}^{1}$. 
\end{example}
\bibliographystyle{amsplain}

\end{document}